\documentclass[12pt]{elsarticle}
\usepackage{lineno,hyperref}
\usepackage{amsmath,amssymb,amsthm}
\usepackage{txfonts}
\usepackage{amssymb}
\usepackage[latin1]{inputenc}
\usepackage{version,tabularx,multicol}
\usepackage{graphicx,float,psfrag}
\usepackage{stmaryrd}
 \usepackage{etoolbox,xstring,mfirstuc,textcase}
 \theoremstyle{plain}
\newtheorem{theorem}{Theorem}[section]

 \newtheorem{proposition}[theorem]{Proposition}

 \newtheorem{lemma}[theorem]{Lemma}

\newtheorem{remark}{Remark}[section]

\allowdisplaybreaks[4]

\newcommand{\ind}{{\bf 1}}
  
  \makeatletter
  \@addtoreset{equation}{section}
  \makeatother

 \def\beqlb{\begin{eqnarray}}\def\eeqlb{\end{eqnarray}}
 \def\beqnn{\begin{eqnarray*}}\def\eeqnn{\end{eqnarray*}}

 \def\qed{\hfill$\Box$\medskip}

\newcommand{\bcen}{\begin{center}}
\newcommand{\ecen}{\end{center}}
\newcommand{\bgeqn}{\begin{equation}}
\newcommand{\edeqn}{\end{equation}}

\modulolinenumbers[5]

\journal{Journal of \LaTeX\ Templates}

\begin{document}

\begin{frontmatter}

\title{On the empty balls of a critical super-Brownian motion}
\tnotetext[mytitlenote]{Jie Xiong's research is supported in part by NSFC grants 61873325, 11831010 and
Southern University of Science and Technology Start up found Y01286120.}

\author[mymainaddress]{Jie xiong}
\ead{xiongj@sustech.edu.cn}

\author[mysecondaryaddress]{Shuxiong Zhang \corref{mycorrespondingauthor}}
\cortext[mycorrespondingauthor]{Corresponding author}
\ead{shuxiong.zhang@mail.bnu.edu.cn}
\address[mymainaddress]{Department of Mathematics and SUSTech International center for Mathematics, Southern University of
Science and Technology, Shenzhen, China}
\address[mysecondaryaddress]{Department of Mathematics, Southern University of
Science and Technology, Shenzhen, China}


\begin{abstract}

Let $\{X_t\}_{t\geq0}$ be a $d$-dimensional critical super-Brownian motion started from a Poisson random measure whose intensity is the Lebesgue measure. Denote by $R_t:=\sup\{u>0: X_t(\{x\in\mathbb{R}^d:|x|< u\})=0\}$ the radius of the largest empty ball centered at the origin of $X_t$. In this work, we prove that for $r>0$,
$$\lim_{t\to\infty}\mathbb{P}\left(\frac{R_t}{t^{(1/d)\wedge(3-d)^+}}\geq r\right)=e^{-A_d(r)},$$
where $A_d(r)$  satisfies $\lim_{r\to\infty}\frac{A_d(r)}{r^{|d-2|+d\ind_{\{d=2\}}}}=C$ for some $C\in(0,\infty)$ depending only on $d$.
\end{abstract}


\end{frontmatter}

\section{Introduction and Main results}
\subsection{Introduction}
   In this work, we consider a $d$-dimensional measure-valued Markov process $\{X_t\}_{t\geq0}$ , called super-Brownian motion (henceforth SBM). For convenience of the reader, we give a brief introduction to the SBM and some pertinent results needed in this article.
   \par

   To characterize the SBM, we first introduce some notations. Let $p>d$ be a constant. Define $\phi_p(x):=(1+|x|^2)^{-p/2},~x\in\mathbb{R}^d$. Denote by
   $$M_p:=\left\{\mu~\text{is a locally finite measure on}~\mathbb{R}^d:\int_{\mathbb{R}^d}\phi_p(x)\mu(dx)<\infty\right\}$$
   the space of $p$-tempered measures; see Etheridge \cite[p23]{Etheridge}. We equip $M_p$ with the topology such that $\mu_n$ converges to $\mu$ in $M_p$ if and only if
   $$\lim_{n\to\infty}\int_{\mathbb{R}^d}f(x)\mu_n(dx)=\int_{\mathbb{R}^d}f(x)\mu(dx),~\forall f\in\{g+\alpha\phi_p:\alpha\in\mathbb{R}, g\in C_c(\mathbb{R}^d)\},$$
   where $C_c(\mathbb{R}^d)$ stands for the class of continuous functions of compact support in $\mathbb{R}^d$. In this paper, for a measure $\mu$, we always use $\mathbb{E}_{\mu}$ to denote the expectation with respect to $\mathbb{P}_{\mu}$, the probability measure under which the SBM has initial value $X_0=\mu$. Let $\psi$ be a function of the form
   $$\psi(u)=a u+b u^2+\int_{(0,\infty)}\left(e^{-ru}-1+ru\right)n(dr),~u\geq0,$$
   where $a\in\mathbb{R}$, $b\geq0$ and $n$ is a $\sigma$-finite measure on $(0,\infty)$ such that $\int_{(0,\infty)}(r\wedge r^2) n(dr)<\infty$.
   \par
    The SBM with initial value $\mu\in M_p$ and branching mechanism $\psi$ is a measure-valued process, whose transition probabilities  are characterized through their Laplace transforms. For any $\mu\in M_p$ and nonnegative continuous function $\phi$ satisfying $\sup_{x\in\mathbb{R}^d}\frac{\phi(x)}{\phi_p(x)}<\infty$, we have
   \begin{align}\label{ujytas}
  \mathbb{E}_\mu\left[e^{-<X_t,\phi>}\right]=e^{-\int_{\mathbb{R}^d} u(t,x)\mu(dx)},
   \end{align}
  where $<X_t,\phi>:=\int_{\mathbb{R}^d}\phi(x)X_t(dx)$ and $u(t,x)$ is the unique positive solution to the following nonlinear partial differential equation:
  \begin{align}
   \begin{cases}
   \frac{\partial u(t,x)}{\partial t}=\frac{1}{2}\Delta u(t,x)-\psi(u(t,x)),\cr
   u(0,x)=\phi(x).\nonumber
   \end{cases}
   \end{align}
In above, $\Delta u(t,x):=\sum^d_{i=1}\frac{\partial^2 u(t,x)}{\partial x_i^2}$ is the Laplace operator. $\{X_t\}_{t\geq0}$ is called a supercritical (critical, subcritical) SBM if $a<0~(=0, >0).$ Our works only consider a typical critical branching mechanism, called binary branching, which is given by
   $$\psi(u)=u^2.$$
In this case, the partial differential equation above is reduced to
  \begin{align}\label{iu676}
   \begin{cases}
   \frac{\partial u(t,x)}{\partial t}=\frac{1}{2}\Delta u(t,x)-u^2(t,x),\cr
   u(0,x)=\phi(x).
   \end{cases}
   \end{align}
Let $\{W_t\}_{t\geq 0}$ be a $d$-dimensional standard Brownian motion. Note that the partial differential equation above is equivalent to the integral equation:
  \begin{align}\label{rfwerg}
  u(t,x)+\int^t_0\mathbb{E}_x\left[u^2(t-s,W_s)\right]ds=\mathbb{E}_x\left[\phi(W_t)\right],
  \end{align}
where for $x\in\mathbb{R}^d$, $\mathbb{E}_x$ stands for the expectation with respect to the probability $\mathbb{P}_{x}$, the probability measure under which $\{W_t\}_{t\geq 0}$ starts from $x$. Furthermore, one can also use the martingale problem to characterize the SBM; see Perkins \cite[p159]{perkins}. We refer the reader to Etheridge \cite{Etheridge}, Perkins \cite{perkins}, Le Gall
\cite{LeGall} and Li \cite{Li} for a more detailed overview to SBM.
\par
 In our work, we consider the SBM starts from the Poisson random measure whose intensity is the Lebeguse measure $\lambda$ on $\mathbb{R}^d$ ($\text{PRM}(\lambda)$ for short). Namely, for any Borel measurable set $A\subset\mathbb{R}^d$,
$$\mathbb{P}_{\text{PRM}(\lambda)}(X_0(A)=k)=\frac{\lambda^k(A)}{k!}e^{-\lambda(A)},~k\geq0.$$
For ease of notation, we write $\mathbb{P}:=\mathbb{P}_{\text{PRM}(\lambda)}$.
\par
For $u>0$, let $B(u):=\{x\in\mathbb{R}^d:|x|<u\}$ be the $d$-dimensional ball with radius $u$ and center at the origin. Write
$$R_t:=\sup\{u>0: X_t(B(u))=0\},$$
with the convention $\sup\emptyset=0$. In other words, $R_t$ is the radius of the largest ball around the origin which does not contain any mass at time $t$ and $B(R(t))$ is the largest empty ball.
\par
This paper aims at showing that after suitable renormalization, $R_t$ converges in distribution to some non-degenerate limit as $t\to\infty$. We note that the renormalization scale depends on the dimension.
\par
The research on the empty ball was first conducted by R\'ev\'esz \cite{reves02} for the critical branching Wiener process model started from $\text{PRM}(\lambda)$. This model, denoted by $\{Z_n\}_{n\geq0}$, is defined as follows. At time $0$, there exist infinite many particles distributed according to $\text{PRM}(\lambda)$. Then, these particles move independently according to the standard normal distribution in unit time. Afterwards, each particle produces children independently according to the Bernoulli distribution $\xi$ in a instant, where $\mathbb{P}(\xi=0)=\mathbb{P}(\xi=1)=1/2$. This forms a random measure at time $1$, denoted by $Z_1$. Similarly, for $n\geq 2$, each particle at time $n-1$, starting from where its parent die, executes a displacement according to the standard normal distribution during time $n-1$ to time $n$ and afterwards executes a reproduction instantly according to $\xi$. This forms a random measure at time $n$, denoted by $Z_n$.
 \par
 Let
 \begin{align}\label{34redfr}
R(n):=\sup\{u>0: Z_n(B(u))=0\}.
\end{align}
 R\'ev\'esz \cite{reves02} proved that $R(n)/n$ converges in distribution to an exponential distribution for the case $d=1$. For $d\geq2$, he presented following two conjectures (see  R\'ev\'esz \cite[Conjecture 1]{reves02} ):\\
 (i) If $d=2$, then for any $r\in(0,\infty)$,
$$\lim_{n\to\infty}\mathbb{P}\left(\frac{R(n)}{\sqrt n}\geq r\right)=e^{-F_2(r)}\in(0,1),$$
where $F_2(r)$ satisfies
$$\lim_{r\to\infty} \frac{F_2(r)}{\pi r^2}=1;$$
(ii) If $d\geq 3$, then for any $r\in(0,\infty)$,
\begin{align}\label{4redff}
\lim_{n\to\infty}\mathbb{P}(R(n)\geq r)=e^{-F_d(r)}\in(0,1),
\end{align}
where $F_d(r)$ satisfies
\begin{align}\label{4redff12}
\lim_{n\to\infty}\frac{F_d(r)}{C_d r^{d-2}}=1,
\end{align}
and $\lim_{d\to\infty}C_d/\big[\frac{\pi^{d/2}}{\Gamma(d/2+1)}\big]^{(d-2)/d}=1$ ($\Gamma(\cdot)$ is the Gamma function). Later, Hu \cite{hu05} partially confirmed R\'ev\'esz's conjecture for $d\geq3$ by showing that $\lim_{n\to\infty}\mathbb{P}(R(n)\geq r)$ exists in $(0,1)$. But (\ref{4redff12}) remains unproven.
\par
Our work gives complete weak convergence results of $R_t$, and theses results are consistent with R\'ev\'esz's conjectures. In $d=1$, we use the modulus of continuity for SBM and Markov property of historical SBM to show that $R_t/t$ converges in law. For $d=2$, by using the scaling property of SBM, we obtain $R_t/\sqrt{t}$ converges in law. For $d\geq3$, we use the mild solution of the PDE (\ref{iu676}) and Feynman-Kac formula to prove that $R_t$ converges in law. Moreover, in our recent study, we believe that our results for the SBM model can facilitate us to solve R\'ev\'esz's conjectures, especially in the case of $d=2$.

\par
 A relevant work to our problem is Zhou \cite{zhou08}. Let $\{K_t\}_{t\geq0}$ be a $d$-dimensional $(1+\beta)$-SBM
and $\tau:=\sup\{t\geq0:K_t(B(g(t)))>0\}$ be the local extinction time, where $\beta\in(0,1]$ and $g(t)\geq0$ is a nondecreasing and right continuous function on $[0,\infty)$. Assume that $d\beta<2$. Zhou proved that
\begin{align}
\mathbb{P}_{\lambda}(\tau<\infty)=
\begin{cases}
1,~\text{if} \int^{\infty}_1g^d(y)y^{-(1+\frac{1}{\beta})}dy<\infty;\cr
0,~\text{otherwise}.\nonumber
\end{cases}
\end{align}
This result implies that in our setting, if $d=1$ and $X_0=\lambda$, then the leading order of $R_t$ is $t$ as $t\to\infty$. Conversely, by studying $R_t$ in the case of $d\geq 2$, one can also obtain some results of the local extinction time of $(1+\beta)$-SBM in the case of $d\beta\geq2$.

\par
  We also mention that in the last few decades, limit theory of SBMs concerning its local behaviours has been studied intensively. For example, Iscoe \cite{Iscoe} studied the decay rate of the hitting probability $\mathbb{P}_{\delta_x}(\exists t>0, X_t(B(r))>0)$ as $|x|\to\infty$. Dawson et al. \cite{DIP1989} considered decay rates of the probabilities
$$\mathbb{P}_{m}(X_t(B(x,r))>0),~\mathbb{P}_{m}(\exists t\geq \delta~\text{s.t.}~X_t(B(x,r))>0 )$$
~\text{and}
$$\mathbb{P}_{m}(\exists t\geq 0~\text{s.t.}~X_t(B(x,r))>0 )~\text{as}~r\to 0,$$
where $\delta,~r>0$, $x\in\mathbb{R}^d$, $B(x,r):=\{y\in\mathbb{R}^d:|y-x|\leq r\}$ and $m$ is a finite measure on $\mathbb{R}^d$. Namely, the probabilities above consider the SBM hits an arbitrarily small ball. Note that
$$\mathbb{P}(R_t\geq r)=1-\mathbb{P}(X_t(B(r))>0).$$
Their works can also give some results for $R_t$. However, since they assumed $X_0=m$ is a finite measure, the SBM will die out in finite time. This implies $R_t=\infty$ for large $t$. Therefore, to make this question meaningful, we consider the SBM starts from an infinite measure $\text{PRM}(\lambda)$. For the maximum of supercritical SBM, see Kyprianou et al. \cite{Kyprianoua} and Pinsky \cite{Pinsky95}. Moreover, Ren et al. \cite{Ren21} and Engl\"ander \cite{Englander2004} considered the corresponding large deviation probabilities. In addition, Mueller et al. \cite{perkins17} investigated the left tail probability of the density of the SBM. For the local times of SBM, see Sugitani \cite{Sugitani}, Hong \cite{jieliang18,jieliang19} and Dawson et al. \cite{DVW} and the references therein.

\subsection{Main Results}
We first consider the $1$-dimensional super Brownian motion.
\begin{theorem}\label{thdim1} If $d=1$, then for any $r\in(0,\infty)$,
$$\lim_{t\to\infty}\mathbb{P}\left(\frac{R_t}{t}\geq r\right)=e^{-2r}.$$
\end{theorem}
Intuitively, the reason for an empty ball $B(u)$ ($u>0$ and may depend on $t$) to form  is that an SBM starting away from $B(u)$ is hard to reach $B(u)$, while the SBM starting from the set $B(u)$ of finite measure will die out in finite time. Since in higher dimensions there are more "particles" in $B(u)$, the SBM starting from $B(u)$ will take a longer time to die out. Thus, in higher dimensions, $R_t$ is smaller.
\begin{theorem}\label{thdim2} If $d=2$, then for any $r\in(0,\infty)$,
$$\lim_{t\to\infty}\mathbb{P}\left(\frac{R_t}{\sqrt t}\geq r\right)=e^{-A_2(r)}\in(0,1),$$
where $A_2(r):=-\log\mathbb{P}_{\lambda}(X_1(B(r))=0)$ satisfying 
$$\lim_{r\to\infty} \frac{A_2(r)}{\pi r^2}=1.$$
\end{theorem}
\begin{remark}Recall that $R(n)$ is defined in (\ref{34redfr}). Assume that $d=2$. In \cite[Theorem 2]{reves02}, R\'ev\'esz proved that for any $\varepsilon>0$ there exist $0<c(\varepsilon)<C(\varepsilon)<\infty$ such that
$$c(\varepsilon)e^{-2\pi(1+\varepsilon)r^2}<\liminf_{n\to\infty}\mathbb{P}\left(\frac{R(n)}{\sqrt n}>r\right)\leq \limsup_{n\to\infty}\mathbb{P}\left(\frac{R(n)}{\sqrt n}>r\right)\leq C(\varepsilon)e^{-2\pi(1-\varepsilon)r^2}.$$
\end{remark}
From above two theorems, one can see that $R_t\to\infty$ for $1\leq d<3$. This is because the SBM suffers local extinction in low dimensions. However, since in high dimensions ($d\geq3$), the SBM is persistent (see \cite[p49]{Etheridge}), we don't need  any renormalization for $R_t$.

\begin{theorem}\label{thdim3}If $d\geq 3$, then for any $r\in(0,\infty)$,
$$\lim_{t\to\infty}\mathbb{P}(R_t\geq r)=e^{-\kappa_dr^{d-2}},$$
where $\kappa_d:=-\lim_{t\to\infty}\int_{\mathbb{R}^d}\log \mathbb{P}_{\delta_x}(X_t(B(1))=0) dx\in(0,\infty)$.
\end{theorem}
\begin{remark} For $d\geq 3$, R\'ev\'esz \cite[Theorem 3]{reves02} proved that there exist constants $0<c_d<C_d<\infty$ such that
$$e^{-c_dr^{d-2}}\leq\liminf_{n\to\infty}\mathbb{P}(R(n)>r)\leq \limsup_{n\to\infty}\mathbb{P}(R(n)>r)\leq e^{-C_dr^{d-2}}.$$
Later, Hu \cite{hu05} proved that $\lim_{n\to\infty}\mathbb{P}(R(n)>r)$ exists. However, Hu did not give any explicit expression of the limit.
\end{remark}

\begin{remark}  Although in all theorems above we consider $X_0=\text{PRM}(\lambda)$, it is still true if $X_0$ is the Lebesgue measure on $\mathbb{R}^d$. Moreover, there exists a constant $T>0$ such that $\mathbb{P}_{\lambda}\left(R_t/t\geq r\right)$ is decreasing (w.r.t. $t$) when $d=1$ and $t\geq T$; $\mathbb{P}_{\lambda}\left(R_t/\sqrt t\geq r\right)=\mathbb{P}_{\lambda}\left(R_1\geq r\right)~\text{for}~d=2~\text{and}~t>0;$ $\mathbb{P}_{\lambda}\left(R_t\geq r\right)$ is increasing when $d\geq 3$ and $t\geq T$.
\end{remark}
\par
The rest of this paper is organized as follows. In Section \ref{sechy}, we collect some properties of the total mass process $\{X_t(\mathbb{R}^d)\}_{t\geq0}$ and the historical super-Brownian motion. In Section \ref{sec1}, we use the modulus of continuity for historical SBM to prove that for $d=1$,
$$\lim_{t\to\infty}\mathbb{P}_{\lambda}\left(X_t(B(rt))=0\right)=e^{-2r}.$$
Then, by (\ref{rfwerg}), we argue that, under $\mathbb{P}_{\lambda}$ and $\mathbb{P}_{\text{PRM}(\lambda)}$, $R_t/t$ have the same limit (in the sense of convergence in distribution). Theorem \ref{thdim1} is then proved. In Section \ref{sec2}, we prove Theorem \ref{thdim2} by the scaling property of the $2$-dimensional SBM. We prove Theorem \ref{thdim3} in Section \ref{sec3}. The idea is to use the mild solution of the partial differential equation (\ref{iu676}) to establish the existence of the limit. Then, using the second moment method and \cite[Lemma 3.2]{DIP1989}, we conclude that the limit is non-degenerate.
\section{Preliminaries}\label{sechy}
The total mass process $\{X_t(\mathbb{R}^d)\}_{t\geq0}$ with initial measure $m\in M_p$ is the so-called continuous state branching process with initial value $m(\mathbb{R}^d)$. We refer the reader to \cite{Li20} for a more detailed overview to it. The following lemma considers the Laplace transform and extinction probability of the continuous state branching process, which can be found in \cite[p22]{Etheridge}.
\begin{lemma}\label{extinprob} Let $\theta>0$, $d\geq 1$. Then
$$\mathbb{E}_m\left[e^{-\theta X_t(\mathbb{R}^d)}\right]=e^{-\frac{\theta m(\mathbb{R}^d)}{1+\theta t}}$$
and
$$\mathbb{P}_m(X_t(\mathbb{R}^d)=0)=e^{-\frac{m(\mathbb{R}^d)}{t}},$$
where by convention, if $m(\mathbb{R}^d)=\infty$, then
$$\mathbb{E}_m\left[e^{-\theta X_t(\mathbb{R}^d)}\right]=0.$$
\end{lemma}
\par
In the remaining of this section, we always assume $d=1$. We first give a brief introduction to the historical SBM. For an explicit definition and a more elaborated discussion on it, we refer the reader to \cite[p187]{perkins}. Let $D(\mathbb{R}_+)$ be the space of c$\grave{\text{a}}$dl$\grave{\text{a}}$g paths from $\mathbb{R}_+$ to $\mathbb{R}$ with the Skorokhod topology. For $y_{\cdot}\in D(\mathbb{R}_+)$ and $t\geq0$, let $y^t_{\cdot}=y_{\cdot\wedge t}$. Define the stopped function space
$$\hat{E}:=\{(t,y^t_{\cdot}):t\geq0,~y_{\cdot}\in D(\mathbb{R}_+)\}$$
with the subspace topology it inherits from $\mathbb{R}_+\times D(\mathbb{R}_+)$. Then $\hat{E}$ is a Polish space. Define an $\hat{E}$-valued process $\{\hat{W}_t\}_{t\geq0}$:
$$
\hat{W}_t:=(t,W_{\cdot\wedge t}).
$$
Let $D(\mathbb{R}_+,\hat{E})$ be the space of c$\grave{\text{a}}$dl$\grave{\text{a}}$g paths from $\mathbb{R}_+$ to $\hat{E}$. For $x\in\mathbb{R}$ and $A\in\mathcal{B}(D(\mathbb{R}_+,\hat{E}))$, we define
$$\hat{\mathbb{P}}_x(\hat{W}_\cdot\in A):=\mathbb{P}_x((\cdot,W^{\cdot})\in A).$$
Then $(\{\hat{W}_t\}_{t\geq0}, (\hat{\mathbb{P}}_x)_{x\in\mathbb{R}})$ is an $\hat{E}$-valued Borel strong Markov process. For a space $E$, we use $M_F(E)$ to denote the space of finite measures on $E$. Thus, for $m\in M_F(\mathbb{R})$, we can construct an $M_F(\hat{E})$-valued superprocess $\{Y_t\}_{t\geq0}$ with spatial motion $\{\hat{W}_t\}_{t\geq0}$, binary branching and initial value $m$ (by identifying $\mathbb{R}=\{(0,y^0_{\cdot}): y_{\cdot}\in D(\mathbb{R}_+)\}\subset \hat{E}$, so $M_F(\mathbb{R})\subset M_F(\hat{E})$). Let $\mathbb{Q}_m$ be the corresponding probability measure. The historical SBM $\{H_t\}_{t\geq0}$ (with respect to the original SBM $\{X_t\}_{t\geq0}$ with initial value $m$) is defined by
$$H_t(A):=Y_t(\Pi^{-1}(A)), A\subset D_t(\mathbb{R}_+),$$
where $D_t(\mathbb{R}_+):=\{~y_{\cdot}\in D(\mathbb{R}_+):y_{\cdot}=y^t_{{\cdot}}\}$ and $\Pi((t,y^t_{\cdot}))=y^t_{\cdot}$ is the projection map from $\hat{E}$ to $D(\mathbb{R}_+)$. Moreover, $\{X_t\}_{t\geq0}$ can be obtained through
$$X_t(A):=H_t(\{y_{\cdot}\in D_t(\mathbb{R}_+):y_t\in A\})~\text{for}~A\in\mathcal{B}(\mathbb{R}).$$
\par

Let $S(H_t)$ be the closed support of the random measure $H_t$ and $C(\mathbb{R}_+)$ be the space of continuous functions from $\mathbb{R}_+$ to $\mathbb{R}$. The following lemma gives a uniform modulus of continuity for all the paths in $S(H_t),~t\geq 0$; see \cite[p195]{perkins}.
\begin{lemma}\label{moducontin} Let $c>2$ be a constant. There exists a random variable $\Delta$ such that $\mathbb{Q}_m$ almost surely, for all $t\geq0$,
$$S(H_t)\subset\left\{y_{\cdot}\in C(\mathbb{R}_+):|y_r-y_s|\leq c|(r-s)\log(r-s)|^{1/2}, \forall r,s>0,|r-s|\leq\Delta\right\}.$$
Moreover, there are constants $\rho>0$ depending only on $c$ and $\kappa>0$ depending only on $d,~c$ such that
$$
\mathbb{Q}_m(\Delta\leq r)\leq \kappa m(1)r^{\rho}~\text{for}~r\in[0,1].
$$
\end{lemma}
 For $t\geq0$, define
 $$M^t_F(D(\mathbb{R}_+)):=\left\{\mu\in M_F(D(\mathbb{R}_+)):\mu\big(\{y_{\cdot}\in D(\mathbb{R}_+):y^t_{\cdot}\neq y_{\cdot}\}\big)=0\right\}$$
In fact, $\{H_t\}_{t\geq0}$ is an inhomogeneous Borel strong Markov process. Furthermore, $H_t\in M^t_F(D(\mathbb{R}_+))$. For $\tau\geq0$ and $\mu\in M^\tau_F(D(\mathbb{R}_+))$, write
$\mathbb{Q}_{\tau, \mu}(H_{\tau+t}\in A):=\mathbb{Q}_m(H_{\tau+t}\in A|H_{\tau}=\mu)$ for measurable set $A\subset M^{\tau+t}_F(D(\mathbb{R}_+))$. In other worlds,
under $\mathbb{Q}_{\tau, \mu}$, the historical SBM starts at time $\tau$ with its initial value $\mu$. The following lemma is borrowed from \cite[p194]{perkins}.
\begin{lemma}\label{histoexitin}
If $A$ is a Borel subset of $D(\mathbb{R}_+)$, and $\mu\in M^{\tau}_F(D(\mathbb{R}_+))$, then for any $t>\tau$,
$$\mathbb{\mathbb{Q}}_{\tau,\mu}(H_s(\{y_{\cdot}\in D(\mathbb{R}_+): y^{\tau}_{\cdot}\in A\})=0,~\forall s\geq t)=e^{-\frac{2\mu(A)}{t-\tau}}.$$
\end{lemma}
\section{Proof of Theorem \ref{thdim1}}\label{sec1}
This section is devoted to prove Theorem \ref{thdim1}. We first present a proposition concerning the empty ball of SBM starting from Lebesgue measure $\lambda$. The proof is mainly inspired by Zhou \cite[Lemma 2.2]{zhou08}. As usual, for a measure $\mu$ on some space $E$, we write $\mu(1):=\mu(E)$ for convenience.
%

\begin{proposition}\label{lowbound1}If $d=1$, then for any $r>0$,
$$\lim_{t\to\infty}\mathbb{P}_{\lambda}\left(X_t(B(rt))=0\right)=e^{-2r}.$$
\end{proposition}
\begin{proof}For a Borel measurable set $A\subset\mathbb{R}$, we denote by $\lambda|A:=\lambda(\cdot\cap A)$ the Lebesgue measure restricted to $A$. Fix $\delta\in(1/2,1)$. Observe that
\begin{align}\label{gfhfb}
\mathbb{P}_{\lambda}\left(X_t(B(rt))>0\right)\leq \mathbb{P}_{\lambda|B(rt+t^{\delta})}\left(X_t(B(rt))>0\right)+ \mathbb{P}_{\lambda|B^c(rt+t^{\delta})}\left(X_t(B(rt))>0\right).
\end{align}
For the first term on the r.h.s. of (\ref{gfhfb}), by Lemma \ref{extinprob},
\begin{align}
 \mathbb{P}_{\lambda|B(rt+t^{\delta})}\left(X_t(B(rt))>0\right)&\leq \mathbb{P}_{\lambda|B(rt+t^{\delta})}\left(X_t(1)>0\right)\cr
 &=1-\exp\left\{-\frac{1}{t}\lambda(B(rt+t^{\delta}))\right\}\cr
 &=1-\exp\left\{-\frac{2}{t}(rt+t^{\delta})\right\}.
\end{align}
Thus,
\begin{equation}\label{rgegeb}
\limsup_{t\to\infty}\mathbb{P}_{\lambda|B(rt+t^{\delta})}\left(X_t(B(rt))>0\right)\leq1-e^{-2r}.
\end{equation}
\par
Next, we use the continuity modulus of historical SBM to prove that the second term on the r.h.s. of (\ref{gfhfb}) tends to $0$. For $j=0,1,...$, let
$$\lambda_j=:\lambda|\left(B(rt+t^{\delta}-1+t^{\delta(j+1)})-B(rt+t^{\delta}-1+t^{\delta j})\right).$$
So, $\lambda|B^c(rt+t^{\delta})=\sum_{j\geq0}\lambda_j$. Recall that under $\mathbb{\mathbb{Q}}_{m}$, $\{H_t\}_{t\geq0}$ stands for the historical super-Brownian of $\{X_t\}_{t\geq0}$ starting from measure $m\in\mathcal{M}_F(\mathbb{R})$. It is simple to see
\begin{align}\label{uikga}
\mathbb{P}_{\lambda|B^c(rt+t^{\delta})}\left(X_t(B(rt))>0\right)&\leq \sum_{j\geq0}\mathbb{P}_{\lambda_j}\left(\exists s\in[0,t], X_s(B(rt))>0\right)\cr
&\leq\sum_{j\geq0}\mathbb{\mathbb{Q}}_{\lambda_j}\left(\exists s\in[0,t], H_s(A(s,rt))>0\right),
\end{align}
where
$$A(u,v):=\left\{y_{\cdot}\in C(\mathbb{R}_+):\inf_{s\leq u}|y_s|\leq v\right\}~\text{for}~u,~v\geq0.$$
Fix $\bar{\delta}\in(0,2\delta-1)$. For $t>1$, let $l_j:=\left\lfloor t e^{t^{\bar{\delta}(j\vee1)}}\right\rfloor$, $j\geq0$. Define $\mathcal{F}_t:=\sigma(H_s,s\in[0,t])$. By the Markov property of $\{H_t\}_{t\geq0}$ and Lemma \ref{histoexitin},
\begin{align}
&\mathbb{Q}_{\lambda_j}\left[H_{(i+1)t/l_j}(A(it/l_j,rt+1))=0\right]\cr
&=\mathbb{Q}_{\lambda_j}\left[\mathbb{Q}_{it/l_i,H_{it/l_j}}\left[H_{(i+1)t/l_j}(A(it/l_j,rt+1))=0|\mathcal{F}_{it/l_j}\right]\right]\cr
&=\mathbb{Q}_{\lambda_j}\left[\mathbb{Q}_{it/l_i,H_{it/l_j}}\left[H_{(i+1)t/l_j}\left(\Big\{y_{\cdot}\in C(\mathbb{R}_+): y^{it/l_j}_{\cdot}\in A(it/l_j,rt+1) \Big\}\right)=0\Bigg{|}\mathcal{F}_{it/l_j}\right]\right]\cr
&=\mathbb{Q}_{\lambda_j}\left[\mathbb{Q}_{it/l_i,H_{it/l_j}}\left[H_{s}\left(\Big\{y_{\cdot}\in C(\mathbb{R}_+): y^{it/l_j}_{\cdot}\in A(it/l_j,rt+1) \Big\}\right)=0,\right.\right.\cr
&~~~~~~~~\left.\left.~\forall s\geq (i+1)t/l_j\Bigg{|}\mathcal{F}_{it/l_j}\right]\right]\cr
&=\mathbb{Q}_{\lambda_j}\left[\exp\left\{-\frac{2H_{it/l_j}\left(A(it/l_j,rt+1)\right)}{t/l_j}\right\}\right].
\end{align}
Thus, for any $1\leq i\leq l_j-1$, we have
\begin{align}\label{oyaef}
\mathbb{Q}_{\lambda_j}\left[H_{(i+1)t/l_j}(A(it/l_j,rt+1))>0\right]&=1-\mathbb{Q}_{\lambda_j}\left[\exp\left\{-\frac{2}{t/l_j}H_{it/l_j}\left(A(it/l_j,rt+1)\right)\right\}\right]\cr
&\leq \frac{2l_j}{t}\mathbb{Q}_{\lambda_j}\left[H_{it/l_j}(A(it/l_j,rt+1))\right]\cr
&=\frac{2l_j}{t}\int_{\mathbb{R}}\mathbb{Q}_{\delta_x}\left[H_{it/l_j}(A(it/l_j,rt+1))\right]\lambda_j(dx)\cr
&=\frac{2l_j}{t}\int_{\mathbb{R}}\mathbb{P}_{x}\left(\inf_{s\leq it/l_j}|W_s|\leq rt+1\right)\lambda_j(dx)\cr
&\leq\frac{2l_j}{t}\int_{\mathbb{R}}\mathbb{P}_{0}\left(\inf_{s\leq it/l_j}(|x|-|W_s|)\leq rt+1\right)\lambda_j(dx)\cr
&\leq\frac{2l_j}{t}\lambda_j(1)\mathbb{P}_{0}\left(\sup_{s\leq it/l_j}|W_s|> t^{\delta}-2+t^{\delta j}\right),
\end{align}
where the third equality follows from the fact that the mean measure of $H_t$ is the Wiener measure stopped at time $t$ (i.e. one moment formula of $H_t$; see \cite[p191, II.8.5]{perkins}). It is simple to see that for any $t>4$ (thus $t^\delta>2$) and $j\geq0$,
\begin{align}\label{uykty}
&\mathbb{P}_{0}\left(\sup_{s\leq it/l_j}|W_s|> t^{\delta}-2+t^{\delta j}\right)\cr
&=\mathbb{P}_{0}\left(\Big|\sup_{s\leq it/l_j}W_s\Big|> t^{\delta}-2+t^{\delta j}~\text{or}~\Big|\inf_{s\leq it/l_j}W_s\Big|> t^{\delta}-2+t^{\delta j}\right)\cr
&\leq2\mathbb{P}_{0}\left(\sup_{s\leq it/l_j}W_s> t^{\delta}-2+t^{\delta j}\right)\cr
&=2\mathbb{P}_{0}\left(\big|W_{it/l_j}\big|> t^{\delta}-2+t^{\delta j}\right)\cr
&=4\mathbb{P}_{0}\left(\frac{W_{it/l_j}}{\sqrt {it/l_j} }> \frac{t^{\delta}-2+t^{\delta j}}{\sqrt {it/l_j}}\right)\cr
&=4\mathbb{P}_{0}\left(\frac{W_{it/l_j}}{\sqrt {it/l_j} }>t^{\left(\delta(j\vee1)-\frac{1}{2}\right)}/2\right)\cr
&\leq \frac{8}{\sqrt{2\pi}}{t^{-(\delta(j\vee1)-\frac{1}{2})}}\exp\left\{-\frac{t^{2\delta (j\vee1)-1}}{8}\right\},
\end{align}
where the second equality follows from the reflection property of the Brownian motion (see \cite[p45]{Peres}) and the last inequality comes from the following classical estimate for standard normal random variable $X$:
$$\mathbb{P}(X>x)\leq \frac{1}{x\sqrt {2\pi}}e^{-\frac{x^2}{2}},~x>0.$$
Plugging (\ref{uykty}) into (\ref{oyaef}) yields that there exists $C_1>0$ depending only on $r,~\delta$ and $\bar{\delta}$ such that for any $t>C_1$ and $j\geq0$,
\begin{align}\label{jkfga}
&\mathbb{Q}_{\lambda_j}\left[H_{(i+1)t/l_j}(A(it/l_j,rt+1))>0\right]\cr
&\leq \frac{16l_j}{t\sqrt{2\pi}}\lambda_j(1){t^{-(\delta(j\vee1)-\frac{1}{2})}}\exp\left\{-\frac{t^{2\delta (j\vee1)-1}}{8}\right\}\cr
&\leq\frac{16e^{t^{\bar{\delta}(j\vee1)}}}{\sqrt{2\pi}}3(r+1)t^{\delta((j\vee1)+1)}{t^{-(\delta(j\vee1)-\frac{1}{2})}}\exp\left\{-\frac{t^{2\delta (j\vee1)-1}}{8}\right\}\cr
&\leq \exp\left\{-\frac{t^{2\delta (j\vee1)-1}}{9}\right\}.
\end{align}
Fix $j\geq0$. Consider $\{X_t\}_{t\geq0}$ starts from $\lambda_j$. Write
$$
T_j:=\inf\{s\geq0:X_s(B(rt))>0\}
$$
with the convention that $\inf \emptyset=+\infty$. Since under the event $\{\exists s\in[0,t], H_s(A(s,rt))>0\}$, we have $0<T_j\leq t$, there exists an integer $i\in\{0,1,...,l_j-1\}$ such that $T_j\in (it/l_j,(i+1)t/l_j]$. For $i\geq 1$, if the SBM
$\{X_t\}_{t\geq0}$ has charged the set $B(rt+1)$ at time $(i-1)t/l_j$ then $H_{it/l_j}(A((i-1)t/l_j,rt+1))>0$. Otherwise, it has not charge $B(rt+1)$, then the support process for $\{X_t\}_{t\geq0}$ has to travel a distance of at least $1$ on time interval $[(i-1)t/l_j,T_j]\subset[(i-1)t/l_j,(i+1)t/l_j]$, which implies $\Delta<2t/l_j$ (use the modulus continuity and the fact $2t/l_j\to0$). For $i=0$, since
$$X_0=\lambda_j=\lambda|\left(B(rt+t^{\delta}-1+t^{\delta(j+1)})-B(rt+t^{\delta}-1+t^{\delta j})\right),$$
the support process for $\{X_t\}_{t\geq0}$ has to travel a distance of at least $t^{\delta}-2+t^{\delta j}$ on time interval $[0,t/l_j]$, which also implies $\Delta<2t/l_j$. Putting these together,
there exists $C_2>0$ depending only on $r,~\delta,~\bar{\delta},~\kappa$ and $\rho$ such that for any $t>C_2$ and $j\geq0$,
\begin{align}
&\mathbb{Q}_{\lambda_j}(\exists s\in[0,t], H_s(A(s,rt))>0)\cr
&\leq \sum^{l_j-1}_{i=1}\mathbb{Q}_{\lambda_j}\left[H_{(i+1)t/l_j}(A(it/l_j,rt+1))>0\right]+\mathbb{Q}_{\lambda_j}(\Delta\leq 2t/l_j)\cr
&\leq l_j\exp\left\{-\frac{t^{2\delta (j\vee1)-1}}{9}\right\}+\lambda_j(1)\kappa2^{\rho}\exp\left\{-\rho t^{\bar{\delta}(j\vee1)}\right\}\cr
&\leq te^{t^{\bar{\delta}(j\vee1)}}\exp\left\{-\frac{t^{2\delta (j\vee1)-1}}{9}\right\}+3(r+1)t^{2\delta(j\vee1+1)}\kappa2^{\rho}\exp\left\{-\rho t^{\bar{\delta}(j\vee1)}\right\}\cr
&\leq\exp\left\{-\rho t^{\bar{\delta}(j\vee1)}/2\right\},
\end{align}
where the second inequality follows from (\ref{jkfga}) and Lemma \ref{moducontin}. Plugging above into (\ref{uikga}) yields that

\begin{align}
\mathbb{P}_{\lambda|B^c(rt+t^{\delta})}\left(X_t(B(rt))>0\right)&\leq\sum_{j\geq 2}\exp\left\{-\rho t^{\bar{\delta}j}/2\right\}+2\exp\left\{-\rho t^{\bar{\delta}}/2\right\}\cr
&\leq\int^{\infty}_1 \exp\left\{-\rho t^{\bar{\delta}x}/2\right\}dx+2\exp\left\{-\rho t^{\bar{\delta}}/2\right\}\cr
&=\frac{1}{\bar{\delta}\log t}\int^{\infty}_{\rho t^{\bar{\delta}}/2} \frac{e^{-u}}{u}du+2\exp\left\{-\rho t^{\bar{\delta}}/2\right\}\cr
&\leq \left[\frac{2}{\rho t^{\bar{\delta}}\bar{\delta}\log t}+2\right]\exp\left\{-\rho t^{\bar{\delta}}/2\right\}.
\end{align}
Thus,
\begin{equation}\label{ytjry12}
\lim_{t\to\infty}e^{\rho t^{\bar{\delta}}/3}\mathbb{P}_{\lambda|B^c(rt+t^{\delta})}\left(X_t(B(rt))>0\right)=0.
\end{equation}
This, combined with (\ref{gfhfb}) and (\ref{rgegeb}), gives
\begin{equation}\label{kigds}
\liminf_{t\to\infty}\mathbb{P}_{\lambda}\left(X_t(B(rt))=0\right)\geq e^{-2r}.
\end{equation}
\par
Next, we deal with the upper bound. By the first moment formula of SBM (see \cite[p38]{LeGall}), we have for any $\varepsilon\in(0,r)$,
\begin{align}
\mathbb{E}_{\lambda|B(rt+t^{\delta})}\left[\frac{1}{t}X_t(B^c(rt))\right]&=\frac{1}{t}\int_{|x|\leq rt+t^{\delta}}\mathbb{P}_0(|x+W_t|>rt)dx\cr
&=\frac{1}{t}\int_{|x|\leq (r-\varepsilon)t}\mathbb{P}_0(|W_t|>\varepsilon t)dx+\frac{1}{t}\int_{(r-\varepsilon)t<|x|\leq rt+t^{\delta}}1dx\cr
&\leq 2\mathbb{P}_0(|W_t|>\varepsilon t)(r-\varepsilon)+2(\varepsilon t+t^{\delta})/t.\nonumber
\end{align}
By the law of large numbers,
$$
\limsup_{t\to\infty}\mathbb{E}_{\lambda|B(rt+t^{\delta})}\left[\frac{1}{t}X_t(B^c(rt))\right]\leq2\varepsilon.
$$
Letting $\varepsilon\to0$ yields
\begin{equation}\label{iouad43}
\lim_{t\to\infty}\mathbb{E}_{\lambda|B(rt+t^{\delta})}\left[\frac{1}{t}X_t(B^c(rt))\right]=0.
\end{equation}
By the branching property of SBM, for any $\eta>0$,
\begin{align}
&\lim_{t\to\infty}\mathbb{E}_{\lambda}\left[e^{-\eta t^{-1}X_t(B(rt))}\right]\cr
&=\lim_{t\to\infty}\mathbb{E}_{\lambda|B^c(rt+t^{\delta})}\left[e^{-\eta t^{-1}X_t(B(rt))}\right]\mathbb{E}_{\lambda|B(rt+t^{\delta})}\left[e^{-\eta t^{-1}X_t(B(rt))}\right]\cr
&=\lim_{t\to\infty}\mathbb{E}_{\lambda|B(rt+t^{\delta})}\left[e^{-\eta t^{-1}X_t(B(rt))}\right],\nonumber
\end{align}
where the last inequality is because by (\ref{ytjry12}),
$$\lim_{t\to\infty}t^{-1}X_t(B(rt))=0,~\mathbb{P}_{\lambda|B^c(rt+t^{\delta})}\text{-in distribution}.$$
 On the other hand, from (\ref {iouad43}), $\frac{1}{t}X_t(B^c(rt))$ converges in distribution to $0$ (under $\mathbb{P}_{\lambda|B(rt+t^{\delta})}$). Furthermore, by Lemma \ref{extinprob}, we have $t^{-1}X_t(1)$ converges in distribution (under $\mathbb{P}_{\lambda|B(rt+t^{\delta})}$). Therefore, under $\mathbb{P}_{\lambda|B(rt+t^{\delta})}$, $t^{-1}X_t(1)-t^{-1}X_t(B^c(rt)$ converges in distribution. Hence,
\begin{align}
\lim_{t\to\infty}\mathbb{E}_{\lambda}\left[e^{-\eta t^{-1}X_t(B(rt))}\right]&=\lim_{t\to\infty}\mathbb{E}_{\lambda|B(rt+t^{\delta})}\left[e^{-\eta [t^{-1}X_t(1)-t^{-1}X_t(B^c(rt))]}\right]\cr
&=\lim_{t\to\infty}\mathbb{E}_{\lambda|B(rt+t^{\delta})}\left[e^{-\eta t^{-1}X_t(1)}\right]\cr
&=\lim_{t\to\infty}\exp\left\{-\frac{\eta t^{-1}\lambda(B(rt+t^{\delta}))}{1+\eta}\right\}\cr
&=e^{-\frac{2r\eta}{1+\eta}}.
\end{align}
Thus, for any $\eta>0$
 \begin{align}
 \limsup_{t\to\infty}\mathbb{P}_{\lambda}(X_t(B(rt))=0)&=\limsup_{t\to\infty}\mathbb{P}_{\lambda}(t^{-1}X_t(B(rt))=0)\cr
 &\leq\lim_{t\to\infty}\mathbb{E}_{\lambda}[e^{-\eta t^{-1}X_t(B(rt))}]\cr
 &=e^{-\frac{2r\eta}{1+\eta}}.
 \end{align}
 Finally,
$$\limsup_{t\to\infty}\mathbb{P}_{\lambda}(X_t(B(rt))=0)\leq \lim_{\eta\to\infty}e^{-\frac{2r\eta}{1+\eta}}\leq e^{-2r}.$$
This, combined with (\ref{kigds}), concludes the proposition.
\end{proof}
\par
Now, we are ready to present the proof of Theorem \ref{thdim1}. Namely, we are going to prove
$$\lim_{t\to\infty}\mathbb{P}\left(\frac{R_t}{t}\geq r\right)=e^{-2r}~\text{for} ~r>0.$$
The idea of the proof is to use the integral equation (\ref{rfwerg}) to argue that $\mathbb{P}\left(\frac{R_t}{t}\geq r\right)$ and $\mathbb{P}_{\lambda}\left(\frac{R_t}{t}\geq r\right)$ have the same asymptotics. We then use Proposition \ref{lowbound1} to conclude the theorem.

\textbf{Proof of Theorem \ref{thdim1}}. For a point measure $F$, write $u\in F$ if $u$ is an atom of $F$. Recall that $X_0=\sum_{u\in X_0}\delta_{u}$ is a Poisson random measure with intensity measure $\lambda$. Let $\{X^{\delta_{u}}_t\}_{t\geq0}$ be the SBM started from $\delta_{u}$ (i.e. a single particle at position $u$). By the branching property,
\begin{align}\label{pos}
\mathbb{P}\left(\frac{R_t}{t}\geq r\right)&=\mathbb{P}\left(X_t(B(tr))=0\right)\cr
&=\mathbb{P}\left(\forall u\in X_0,X^{\delta_{u}}_t(B(tr))=0\right)\cr
&=\mathbb{E}\left[\Pi_{u\in X_0}\mathbb{P}_{\delta_{u}}(X_t(B(tr))=0)\right]\cr
&=\mathbb{E}\left[e^{\sum_{u\in X_0}\log \mathbb{P}_{\delta_{u}}(X_t(B(tr))=0)}\right]\cr
&=\mathbb{E}\left[e^{-\int_\mathbb{R}-\log \mathbb{P}_{\delta_{x}}(X_t(B(tr))=0)X_0(dx)}\right]\cr
&=e^{-\int_\mathbb{R} \mathbb{P}_{\delta_{x}}(X_t(B(tr))>0)dx },
\end{align}
where the last equality follows from the Laplace  functional formula of Poisson random measures (see \cite[p19, (2.17)]{Bovier}).
\par
For $\theta>0$, let
$$u_{\theta}(t,x):=-\log\mathbb{E}_{\delta_{x}}\left[e^{-<X_t,\theta \ind_{B(tr)}>}\right],~t\geq0,~x\in\mathbb{R}.$$
Since $u_{\theta}(t,x)$ is increasing w.r.t. $\theta$, $u(t,x):=\lim_{\theta\to\infty}u_{\theta}(t,x)$ exists. Therefore,
\begin{align}\label{pos1}
\mathbb{P}_{\delta_{x}}(X_t(B(tr))=0)&=\lim_{\theta\to\infty}\mathbb{E}_{\delta_{x}}\left[e^{-<X_t,\theta \ind_{B(tr)}>}\right]\cr
&=\lim_{\theta\to\infty}e^{-u_{\theta}(t,x)}\cr
&=e^{-u(t,x)}.
\end{align}
From (\ref{rfwerg}), we have
$$u_{\theta}(t,x)+\int^t_0\mathbb{E}_{x}\left[u^2_{\theta}(t-s,W_s)\right]ds=\mathbb{E}_{x}\left[\theta\ind_{B(tr)}(W_t)\right].$$
Integrating w.r.t. $x$ gives
$$\int_\mathbb{R}u_{\theta}(t,x)dx+\int^t_0\int_\mathbb{R}u^2_{\theta}(t-s,y)dyds=2r\theta t.$$
Set
$$G_{\theta}(t):=\int_\mathbb{R}u_{\theta}(t,x)dx=2r\theta t-\int^t_0\int_\mathbb{R}u^2_{\theta}(s,y)dyds.$$
From Lemma \ref{extinprob}, for any $\theta,~t>0$,
\begin{align}\label{3ewedr2ws}
u_{\theta}(t,x)<\frac{1}{t}.
\end{align}
By (\ref{ujytas}),
\begin{align}\label{4re4ffr3}
\mathbb{P}_{\lambda}(X_t(B(tr))=0)&=\lim_{\theta\to\infty}\mathbb{E}_{\lambda}\left[e^{-<X_t,\theta \ind_{B(tr)}>}\right]\cr
&=\lim_{\theta\to\infty}e^{-\int_{\mathbb{R}} u_{\theta}(t,x)dx}\cr
&=e^{-\int_{\mathbb{R}} u(t,x)dx},
\end{align}
where the last inequality follows from L\'evy's monotone convergence theorem. Thus, by Proposition \ref{lowbound1}, there exists some $C_3>1$ depending only on $r$ such that for $t>C_3$ and $\theta>0$,
$$\int_{\mathbb{R}} u_{\theta}(t,x)dx\leq\int_{\mathbb{R}} u(t,x)dx<3r.$$
Then, for $\theta>\frac{3}{2}$ and $t>C_3$,
\begin{align}\label{wewr}
G'_{\theta}(t)&=2r\theta-\int_\mathbb{R}u^2_{\theta}(t,y)dy\cr
&\geq2r\theta-\frac{1}{t}\int_\mathbb{R}u_{\theta}(t,y)dy\cr
&\geq2r\theta-\int_\mathbb{R}u_{\theta}(t,y)dy\cr
&\geq2r\theta-3r\cr
&>0,
\end{align}
where the first inequality follows from (\ref{3ewedr2ws}). Thus, for any $t_2>t_1>C_3$ and $\theta>\frac{3}{2}$,
$$\int_\mathbb{R} u_{\theta}(t_1,x)dx<\int_\mathbb{R} u_{\theta}(t_2,x)dx.$$
Since $\int_\mathbb{R} u(t,x)dx=\lim_{\theta\to\infty}\int_\mathbb{R} u_{\theta}(t,x)dx$, we have

$$\int_\mathbb{R} u(t_1,x)dx\leq\int_\mathbb{R} u(t_2,x)dx.$$
So $\int_\mathbb{R} u(t,x)dx$ is increasing on $(C_3,\infty)$. Thus, by Proposition \ref{lowbound1},
\begin{equation}\label{ikygwd}
\lim_{t\to\infty}e^{-\int_{\mathbb{R}} u(t,x)dx}=\lim_{t\to\infty}\mathbb{P}_{\lambda}(X_t(B(tr))=0)\downarrow e^{-2r}.
\end{equation}
Since $x-\frac{x^2}{2}<1-e^{-x}<x$ for $x>0$, we have
\begin{align}\label{liluada}
\int_{\mathbb{R}} u(t,x)dx&\geq \int_{\mathbb{R}} 1-e^{-u(t,x)}dx\cr
&\geq\int_{\mathbb{R}} u(t,x)-\frac{u^2(t,x)}{2}dx\cr
&\geq\left(1-\frac{1}{2t}\right)\int_{\mathbb{R}} u(t,x)dx.\cr
\end{align}
where the last inequality uses the fact that $u(t,x)=\lim_{\theta\to\infty}u_{\theta}(t,x)\leq\frac{1}{t}$ (see \ref{3ewedr2ws}). This, combined with (\ref{ikygwd}), yields that
\begin{align}\label{oiuwe67}
\lim_{t\to\infty}\int_{\mathbb{R}} 1-e^{-u(t,x)}dx=\lim_{t\to\infty}\int_{\mathbb{R}} u(t,x)dx=2r.
\end{align}
So, by (\ref{pos}) and (\ref{pos1}),
\begin{align}
\lim_{t\to\infty}\mathbb{P}\left(\frac{R_t}{t}\geq r\right)=\lim_{t\to\infty}e^{-\int_\mathbb{R}1-e^{-u(t,x)}dx}=e^{-2r}.\nonumber
\end{align}
We have completed the proof of Theorem \ref{thdim1}.\qed

\section{Proof of Theorem \ref{thdim2}}\label{sec2}
In this section, we are going to prove Theorem \ref{thdim2}. Namely, if $d=2$, then for any $r\in(0,\infty)$,
$$\lim_{t\to\infty}\mathbb{P}\left(\frac{R_t}{\sqrt t}\geq r\right)=e^{-A_2(r)}\in(0,1),$$
where $A_2(r)=-\log\mathbb{P}_{\lambda}(X_1(B(r))=0)$ satisfying
$$\lim_{r\to\infty} \frac{A_2(r)}{\pi r^2}=1.$$
The idea of the proof can be divided into 4 steps:\\
Step 1. By the scaling property, we have
$$\mathbb{P}_{\lambda}\left(X_t\left(B(\sqrt tr)\right)=0\right)=\mathbb{P}_{\lambda}(X_1(B(r))=0)~\text{for}~t>0;$$
Step 2. Use the Feynman-Kac representation to give the desired lower bound of $\mathbb{P}_{\lambda}(X_1(B(r))=0)$;\\
Step 3. From the observation
$$\mathbb{P}_{\lambda}(X_1(B(r))=0)\leq \mathbb{P}_{\lambda|B(r)}(X_1(B(r))=0)$$
 and the extinction probability of $\{X_t(1)\}_{t\geq 0}$, we get the desired upper bound of $\mathbb{P}_{\lambda}(X_1(B(r))=0)$;\\
Step 4. Similar to the proof of Theorem \ref{thdim1}, we obtain
$$\lim_{t\to\infty}\mathbb{P}\left(R_t/\sqrt t\geq r\right)=\lim_{t\to\infty}\mathbb{P}_{\lambda}\left(X_t(B(\sqrt tr))=0\right).$$
\par
\textbf{Proof of Theorem \ref{thdim2}}. \textbf{Step 1}. From \cite[p51]{Etheridge}, we have the following scaling property for 2-dimensional SBM $\{X_t\}_{t\ge0}$ started from Lebesgue measure: for any $\eta,~t>0$ and $A\in\mathcal{B}(\mathbb{R}^2)$,
$$X_t(A)\overset{\text{Law}}=\frac{1}{\eta^2}X_{\eta^2t}(\eta A).$$
Let $\eta=\frac{1}{\sqrt t}$ and $A=B(\sqrt tr)$, then
$$X_t(B(\sqrt tr))\overset{\text{Law}}=\frac{1}{t}X_{1}(B(r)).$$
Thus,
$$\mathbb{P}_{\lambda}\left(X_t\left(B(\sqrt tr)\right)=0\right)=\mathbb{P}_{\lambda}(X_1(B(r))=0).$$
\par
\textbf{Step 2}.
Similar to the case of $d=1$, let
$$u(t,x):=-\log\mathbb{P}_{\delta_x}(X_t(B(r))=0),~t>0,~x\in\mathbb{R}^2.$$
Due to the same reason as (\ref{4re4ffr3}), we have
\begin{equation}\label{kufhn}
\mathbb{P}_{\lambda}(X_1(B(r))=0)=e^{-\int_{\mathbb{R}^2} u(1,x)dx}.
\end{equation}
Thus, to finish Step 2, it suffices to give an upper bound of $u(1,x)$. Fix $\delta>0$, we are going to prove
\begin{equation}\label{ilu57}
M_{\delta}(r):=\sup_{|x|>(1+\delta)r,~t>0}u(t,x)<\infty.
\end{equation}
Suppose that (\ref{ilu57}) is not true. Then, by the fact that $u(t,x)<1/t$, there exist some $|x_n|>(1+\delta)r$, $n\geq 1$ and $t_n\rightarrow0$ such that $u(t_n,x_n)\rightarrow\infty$. Thus,
\begin{align}\label{hjgh34}
\lim_{n\to\infty}\mathbb{P}_{\delta_{x_n}}(X_{t_n}(B(r))>0)&=1-\lim_{n\to\infty}\mathbb{P}_{\delta_{x_n}}(X_{t_n}(B(r))=0)\cr
&=1-\lim_{n\to\infty}e^{-u(t_n,x_n)}\cr
&=1.
\end{align}
On the other hand, since $B(r)\subset B^c(x_n,\delta r)$, we have
\begin{align}\label{738hdsad}
\limsup_{n\to\infty}\mathbb{P}_{\delta_{x_n}}(X_{t_n}(B(r))>0)&\leq\limsup_{n\to\infty}\mathbb{P}_{\delta_{x_n}}(X_{t_n}(B(x_n,\delta r)^c)>0)\cr
&=\limsup_{n\to\infty}\mathbb{P}_{\delta_{0}}(X_{t_n}(B(0,\delta r)^c)>0)\cr
&=0,
\end{align}
where the last equality follows from the continuity modulus of the historical SBM. Thus,  (\ref{738hdsad}) contradicts (\ref{hjgh34}), and therefore (\ref{ilu57}) holds. Furthermore, since $\delta$ and $r$ are arbitrary positive constants, we have
\begin{equation}\label{ilu58}
M(r):=\sup_{|x|\geq r,~t>0}u(t,x)<\infty.
\end{equation}

Let $\phi(x)=\ind_{B(r)}(x)$. For $\theta>0$, let
$$u_{\theta}(t,x):=-\log\mathbb{E}_{\delta_x}\left[e^{-<X_t,\theta \phi>}\right].$$
By the Feynman-Kac representation \cite[p170]{perkins},
$$u_{\theta}(t,x)=\mathbb{E}_x\left[\theta\phi(W_t)e^{-\int^t_0u_{\theta}(t-s,W_s)}\right].$$
Let $T_r:=\inf\{t>0: |W_t|=r\}$.  Then, by the strong Markov property of Brownian motion, for $|x|\geq (1+\delta)r$,
\begin{align}\label{ikqee343}
u_{\theta}(t,x)&=\mathbb{E}_x\left[\ind_{\{T_r<t\}}e^{-\int^{T_r}_0u_{\theta}(t-s,W_s)ds}\mathbb{E}_{W_{T_r}}\left[\theta\phi(W_{t-T_r})e^{-\int^{t-T_r}_0u_{\theta}(t-T_r-s,W_s)ds}\right]\right]\cr
&\leq\mathbb{E}_x\left[\ind_{\{T_r<t\}}u_{\theta}(t-T_r,W_{T_r})\right]\cr
&\leq M(r)\mathbb{P}_{0}\left(\inf_{s\in(0,t)}|x+W_s|\leq r\right),\cr
\end{align}
where the first equality follows from the fact that $\phi(W_t)=0$ for $T_r>t$ and the last inequality follows from (\ref{ilu58}). Let $\{W^1_t\}_{t\geq0}$ be the one-dimensional standard-Brownian motion. Since $\sup _{s\in(0,t)}W^1_s\overset{\text{Law}}=|W^1_t|$, we have
\begin{align}\label{u56jh12}
\mathbb{P}_{0}\left(\inf_{s\in(0,t)}|x+W_s|\leq r\right)&\leq\mathbb{P}_{0}\left(|x|-r\leq\sup_{s\in(0,t)}|W_s|\right)\cr
&\leq 2\mathbb{P}_{0}\left(\sup _{s\in(0,t)}W^1_s\geq(|x|-r)/\sqrt 2\right)\cr
&= 2\mathbb{P}_{0}\left(|W^1_t|/\sqrt t\geq(|x|-r)/\sqrt {2t}\right)\cr
&\leq \frac{4\sqrt t}{\sqrt {\pi}(|x|-r)}e^{-\frac{(|x|-r)^2}{4t}}.
\end{align}
Plugging above into (\ref{ikqee343}) yields that
$$
u_{\theta}(t,x)\leq M(r)\frac{4\sqrt t}{\sqrt {\pi}(|x|-r)}e^{-\frac{(|x|-r)^2}{4t}}.
$$
Since $M(r)$ is decreasing (thus bounded) and $\lim_{\theta\to\infty}u_{\theta}(t,x)=u(t,x)$, there exists some constant $C_4>0$ such that for any $|x|\geq (1+\delta)r$,
$$
u(1,x)\leq M(r)\frac{4}{\sqrt {\pi}(|x|-r)}e^{-\frac{(|x|-r)^2}{4}}<\frac{C_4}{|x|-r}e^{-\frac{(|x|-r)^2}{4}}.
$$

 Thus,
\begin{align}
\int_{\mathbb{R}^2} u(1,x)dx&\leq \int_{|x|\leq (1+\delta)r} 1dx+\int_{|x|>(1+\delta)r}\frac{C_4}{|x|-r}e^{-\frac{(|x|-r)^2}{2}} dx\cr
&\leq\pi(1+\delta)^2r^2+\int_{|x|>(1+\delta)r}\frac{C_4}{|x|-r}e^{-\frac{(|x|-r)^2}{2}} dx\cr
&\leq\pi(1+\delta)^2r^2+\int_{w>\delta r}\frac{C_4}{w}e^{-\frac{w^2}{2}}w^2 dw\cr
&=\pi(1+\delta)^2r^2+C_4e^{-\delta^2r^2/2}.\nonumber
\end{align}
Taking limits yields that
$$\lim_{r\to\infty}\frac{\int_{\mathbb{R}^2} u(1,x)dx}{\pi r^2}\leq 1+\delta.$$
Let $\delta\to0$, then plugging it into (\ref{kufhn}) yields
\begin{equation}\label{iuy9870}
\limsup_{r\to\infty}\frac{-\log \mathbb{P}_{\lambda}(X_1(B(r))=0)}{\pi r^2}\leq1.
\end{equation}
\par
\textbf{Step 3}. Fix $\bar{\delta}\in(0,1)$. It is simple to see that
\begin{align}
\mathbb{P}_{\lambda}(X_1(B(r))=0)&=\exp\left\{\int_{\mathbb{R}^2} \log\mathbb{P}_{\delta_x}(X_1(B(r))=0) dx\right\}\cr
&\leq\exp\left\{\int_{|x|\leq (1-\bar{\delta})r} \log\mathbb{P}_{\delta_x}(X_1(B(r))=0) dx\right\}\cr
&=\exp\left\{\int_{|x|\leq (1-\bar{\delta})r} \log\mathbb{P}_{\delta_0}(X_1(B(x,r))=0) dx\right\}\cr
&\leq\exp\left\{\int_{|x|\leq (1-\bar{\delta})r} \log\mathbb{P}_{\delta_0}(X_1(B(\bar{\delta} r))=0) dx\right\}\cr
&= \exp\left\{\pi(1-\bar{\delta})^2r^2\log\mathbb{P}_{\delta_0}(X_1(B(\bar{\delta} r))=0)\right\}.
\end{align}
By the dominated convergence theorem and Lemma \ref{extinprob},
$$\lim_{r\to\infty}\log\mathbb{P}_{\delta_0}(X_1(B(\bar{\delta} r))=0)=\log\mathbb{P}_{\delta_0}(X_1(\mathbb{R}^2)=0)=-1.$$
Therefore,
$$\liminf_{r\to\infty}\frac{-\log\mathbb{P}_{\lambda}(X_1(B(r))=0)}{\pi r^2}\geq(1-\bar{\delta})^2.$$
The desired lower bound follows by letting $\bar{\delta}\to0$.

\par \textbf{Step 4}. By similar arguments as in the proof of (\ref{pos})-(\ref{oiuwe67}), we obtain that
\begin{align}
\lim_{t\to\infty}\mathbb{P}\left(\frac{R_t}{\sqrt t}\geq r\right)=\lim_{t\to\infty}e^{-\int_{\mathbb{R}^2}1-e^{-u(t,x)}dx}=\lim_{t\to\infty}e^{-\int_{\mathbb{R}^2}u(t,x)dx}.\nonumber
\end{align}
In fact, the mainly changes are to replace $t$, $\mathbb{R}$, Proposition \ref{lowbound1} with $\sqrt t$, $\mathbb{R}^2$,(\ref{iuy9870}), respectively. So, we feel free to omit its details here. Putting all steps together, we get that
\begin{align}
\lim_{t\to\infty}\mathbb{P}\left(\frac{R_t}{\sqrt t}\geq r\right)&=\lim_{t\to\infty}\mathbb{P}_{\lambda}\left(\frac{R_t}{\sqrt t}\geq r\right)\cr
&=\mathbb{P}_{\lambda}\left(X_1(B(r))=0\right)\cr
&=:e^{-A_2(r)}.\nonumber
\end{align}

 We have completed the proof of Theorem \ref{thdim2}.
 \qed

\section{Proof of Theorem \ref{thdim3}}\label{sec3}
In this section, we are going to prove Theorem \ref{thdim3}. Namely, if $d\geq 3$, then for any $r\in(0,\infty)$,
$$\lim_{t\to\infty}\mathbb{P}(R_t\geq r)=e^{-\kappa_dr^{d-2}}\in(0,1).$$
The proof will be divided into 4 steps:\\
Step 1. Using the semigroup property of $u(t,x)$ and the mild form of the PDE (\ref{iu676}) to show that
$$\lim_{t\to\infty}\mathbb{P}_{\lambda}(R_t\geq r)~\text{exists};$$
Step 2. By the scaling property of $u(t,x)$, we obtain
 $$\lim_{t\to\infty}\mathbb{P}(R_t\geq r)=\lim_{t\to\infty}\mathbb{P}_{\lambda}(R_t\geq r)=e^{-\kappa_dr^{d-2}};$$\\
Step 3. Applying \cite[Lemma 3.2]{DIP1989}, we prove $\kappa_d>0$.\\
Step 4. By the second moment method, we show $\kappa_d<\infty$.
\par

\textbf{Proof of Theorem \ref{thdim3}}. \textbf{Step 1.} In this step, we show that $\lim_{t\to\infty}\mathbb{P}_{\lambda}\left(R_t\geq r\right)$ exists. Let $\psi(x)\in C^2_b(\mathbb{R}^d)$ be a non-negative radially symmetric function such that $\{x:\psi(x)>0\}=\{x:|x|\leq r\}$.
Note that
\begin{align}\label{3reff4}
\mathbb{P}_{\delta_{x}}(X_t(B(r))=0)&=\lim_{\theta\to\infty}\mathbb{E}_{\delta_{x}}[e^{-<X_t,\theta \ind_{B(r)}>}]\cr
&=\lim_{\theta\to\infty}\mathbb{E}_{\delta_{x}}\left[e^{-<X_t,\theta\psi(x)>}\right]\cr
&=:\lim_{\theta\to\infty}e^{-u^{(\psi)}_{\theta}(t,x)}\cr
&=:e^{-u(t,x)},
\end{align}
where $u^{(\psi)}_{\theta}(t,x)$ is the unique positive solution to the equation:
\begin{align}
\begin{cases}
\frac{\partial u(t,x)}{\partial t}=\frac{1}{2}\Delta u(t,x)-u^2(t,x),\cr
u(0,x)=\theta\psi(x).
\end{cases}
\end{align}
Fix $t_1\in(0,t)$. From (\ref{rfwerg}), we have
$$
u^{(u^{(\psi)}_{\theta}(t_1,\cdot))}_{\theta}(t-t_1,x)+\int^{t-t_1}_0\mathbb{E}_x\left[\left(u^{(u^{(\psi)}_{\theta}(t_1,\cdot))}_{\theta}(t-t_1-s,W_s)\right)^2\right]ds=\mathbb{E}_x\left[u^{(\psi)}_{\theta}(t_1,W_{t-t_1})\right].
$$
By the semigroup property $u^{(u^{(\psi)}_{\theta}(t_1,\cdot))}_{\theta}(t-t_1,x)=u^{(\psi)}_{\theta}(t,x)$ (see \cite[p32]{LeGall}), we get that
$$
u^{(\psi)}_{\theta}(t,x)+\int^{t-t_1}_0\mathbb{E}_x\left[\left(u^{(\psi)}_{\theta}(t-s,W_s)\right)^2\right]ds=\mathbb{E}_x\left[u^{(\psi)}_{\theta}(t_1,W_{t-t_1})\right].
$$
Since $u^{(\psi)}_{\theta}(t,x)<\frac{1}{t}$, by the dominated convergence theorem,
$$
u(t,x)+\int^{t-t_1}_0\mathbb{E}_x\left[\left(u(t-s,W_s)\right)^2\right]ds=\mathbb{E}_x\left[u(t_1,W_{t-t_1})\right].
$$
We write above into its mild form:
$$u(t,x)=P_{t-t_1}(t_1,x)-\int^t_{t_1}P_{t-s}u^2(s,x)ds.$$
Integrating w.r.t. $x$ and making use of Fubini's theorem, we obtain
\begin{align}
\int_{\mathbb{R}^d} u(t,x)dx&=\int_{\mathbb{R}^d}\left[\int_{\mathbb{R}^d} p_{t-t_1}(x,y)dx\right]u(t_1,y)dy-\int^t_{t_1}\left[\int_{\mathbb{R}^d} \int_{\mathbb{R}^d} p_{t-s}(x,y)dxu^2(s,y)dy\right]ds\cr
&=\int_{\mathbb{R}^d} u(t_1,y)dy-\int^t_{t_1}\int_{\mathbb{R}^d} u^2(s,y)dyds.\nonumber
\end{align}
Thus, $\int_{\mathbb{R}^d} u(t,x)dx$ is decreasing w.r.t. $t$. This, together with (\ref {4re4ffr3}), implies
$$\lim_{t\to\infty}\mathbb{P}_{\lambda}\left(R_t\geq r\right)=\lim_{t\to\infty}e^{-\int_{\mathbb{R}} u(t,x)dx}$$
exists.
\par
\textbf{Step 2.} In this step, we show that $\lim_{t\to\infty}\mathbb{P}\left(R_t\geq r\right)
=e^{-\kappa_dr^{d-2}}$. Let $\phi(x)\in C^2_b(\mathbb{R}^d)$ be a non-negative radially symmetric function such that $\{x:\phi(x)>0\}=\{x:|x|\leq1\}$. Let $\phi_r(x)=\phi(x/r)$. It follows that
\begin{align}
\mathbb{P}_{\delta_{x}}(X_t(B(r))=0)&=\lim_{\theta\to\infty}\mathbb{E}_{\delta_{x}}\left[e^{-<X_t,\theta \ind_{B(r)}>}\right]\cr
&=\lim_{\theta\to\infty}\mathbb{E}_{\delta_{x}}\left[e^{-<X_t,\theta\phi_r(x)>}\right]\cr
&=:\lim_{\theta\to\infty}e^{-u^r_{\theta}(t,x)}\cr
&=e^{-u(t,x)}.
\end{align}
In this step, we write $u^r(t,x):=u(t,x)$ to emphasize that $u(t,x)$ depends on $r$. So,
\begin{equation}\label{987fdvw}
\mathbb{P}_{\delta_{x}}(X_t(B(r))=0)=e^{-u^r(t,x)}.
\end{equation}
Thus, to finish this step it suffices to give an upper bound of $u^r(t,x)$. Note that $u^r_{\theta} (t,x)$ is the unique solution of
\begin{align}
\begin{cases}
\frac{\partial u(t,x)}{\partial t}=\frac{1}{2}\Delta u(t,x)-u^2(t,x),\cr
u(0,x)=\theta\phi_{r}(x).
\end{cases}
\end{align}
Therefore, we have the following scaling property of $u(t,x)$. For $\theta,~\varepsilon>0$,
$$u^r_{\theta} (t,x)={\varepsilon}^{-2}u^{r/\varepsilon}_{\theta {\varepsilon}^2}(t{\varepsilon}^{-2},x{\varepsilon}^{-1}).$$
This yields that
\begin{align}\label{utoy09}
\mathbb{P}_{\lambda}(X_{t}(B(r))=0)&=\lim_{\theta\to\infty}e^{-\int_{\mathbb{R}^d} u^r_{\theta} (t,x)dx}\cr
&=\lim_{\theta\to\infty}e^{-\int_{\mathbb{R}^d} r^{-2}u^1_{\theta r^2}(tr^{-2},xr^{-1})dx}\cr
&=e^{-\int_{\mathbb{R}^d} r^{-2}u^1(tr^{-2},xr^{-1})dx}\cr
&=e^{-r^{d-2}\int_{\mathbb{R}^d} u^1(tr^{-2},x)dx}.
\end{align}
Hence, using the monotonicity of $\int_{\mathbb{R}^d} u(t,x)dx$ and (\ref{3reff4}), we have
$$\lim_{t\to\infty}\mathbb{P}_{\lambda}(X_{t}(B(r))=0)=e^{-\kappa_dr^{d-2}}.$$
This, combined with (\ref{liluada}), yields that
\begin{align}\label{6hdfsr}
\lim_{t\to\infty}\mathbb{P}\left(R_t\geq r\right)&=\lim_{t\to\infty}e^{-\int\left(1-e^{-u(t,x)}\right)dx}\cr
&=\lim_{t\to\infty}e^{-\int u(t,x)dx }\cr
&=\lim_{t\to\infty}\mathbb{P}_{\lambda}\left(R_t\geq r\right)\cr
&=\lim_{t\to\infty}\mathbb{P}_{\lambda}\left(X_t(B(r))=0\right)\cr
&=e^{-\kappa_dr^{d-2}}.
\end{align}
\par
 \textbf{Step 3.}  In this step, we show that $\kappa_d<\infty$. From \cite[Lemma 3.2]{DIP1989}, there exists a constant $C(d)$ depending only on $d$ such that for all $t>1$ and $x\in\mathbb{R}^d$
 $$
 u^1(t,x)<C(d)p(t+1,x),
 $$
 where $p(t,x):=\frac{1}{(2\pi t)^{d/2}}e^{-\frac{|x|^2}{2t}}$ is the density function of the $d$-dimensional Brownian motion. Hence,
 $$\kappa_d=\lim_{t\to\infty}\int_{\mathbb{R}^d}u^1(t,x)dx\leq C(d)<\infty. $$
 \par
\textbf{Step 4.}  In this step, we show that $\kappa_d>0$. In fact, this has been proved in \cite[Lemma 3.3]{DIP1989}. Nevertheless, here we use a different method to prove it. Since
\begin{align}\label{trgetfvda}
\mathbb{P}(R_t\geq r)=e^{-\int\mathbb{P}_{\delta_x}(X_t(B(r))>0)dx},
\end{align}
To prove $\kappa_d>0$, it suffices to get a lower bound of $\mathbb{P}_{\delta_x}(X_t(B(r))>0)$.
\par
By the Paley-Zygmund inequality,
\begin{align}\label{vwevqv}
\mathbb{P}_{\delta_x}(X_t(B(r))>0)&\geq\frac{\mathbb{E}^2_{\delta_x}[X_t(B(r))]}{\mathbb{E}_{\delta_x}[X^2_t(B(r))]}\cr
&=\frac{(P_t\ind_{B(r)}(x))^2}{(P_t\ind_{B(r)}(x))^2+2\int^t_0P_s\left[P_{t-s}\ind_{B(r)}(x)\right]^2ds},
\end{align}
where the equality follows from the moments formula of SBM (see \cite[p38-39]{LeGall}).
\par
In the next, we are going to give a lower bound of $P_t\ind_{B(r)}(x)$. In the following, we assume $t\geq r^2$. Observe that
\begin{align}
e^{-\frac{2r|x|}{2t}}\geq
\begin{cases}
e^{-\frac{|x|^2}{t}},&~|x|\geq r;\cr
e^{-1},&~|x|<r.\nonumber
\end{cases}
\end{align}
Thus,
\begin{equation}\label{dsfse12}
e^{-\frac{2r|x|}{2t}}\geq e^{-1}e^{-\frac{|x|^2}{t}}.
\end{equation}
Observe that if $|y|\leq r$, then
$$|y-x|^2\leq (|x|+|y|)^2\leq |x|^2+2r|x|+r^2.$$
Let $v_d(r)$ be the volume of $d$-dimensional ball with radius $r$. Through simple calculations, we have
\begin{align}\label{fhegdb}
P_t\ind_{B(r)}(x)&=\mathbb{P}_0(|W_t+x|\leq r)\cr
&=\int_{|y|\leq r}\frac{1}{(2\pi t)^{d/2}}e^{-\frac{|y-x|^2}{2t}}dy\cr
&\geq\int_{|y|\leq r}\frac{1}{(2\pi t)^{d/2}}e^{-\frac{|x|^2}{2t}}e^{-\frac{r^2+2r|x|}{2t}}dy\cr
&\geq\int_{|y|\leq r}\frac{1}{(2\pi t)^{d/2}}e^{-\frac{|x|^2}{2t}}e^{-1/2}e^{-\frac{2r|x|}{2t}}dy\cr
&\geq\int_{|y|\leq r}\frac{1}{(2\pi t)^{d/2}}e^{-\frac{|x|^2}{2t}}e^{-3/2}e^{-\frac{|x|^2}{t}}dy\cr
&\geq e^{-3/2}\int_{|y|\leq r}3^{-d/2}\frac{1}{(2\pi t/3)^{d/2}}e^{-\frac{|x|^2}{2t/3}}dy\cr
&\geq e^{-3/2}3^{-d/2}v_d(1)r^dp(t/3,x),\cr
\end{align}
where the third inequality follows from (\ref{dsfse12}).
\par
In the next, we give an upper bound of $\int^t_0 P_s\left[P_{t-s}\ind_{B(r)}(x)\right]^2ds$. Note that for $s\in(0,t)$,
\begin{align}
P_{t-s}\ind_{B(r)}(x)&=\mathbb{P}_0\left(|x+W_{t-s}|\leq r\right)\cr
&=\int_{|y|\leq r}\frac{1}{(2\pi(t-s))^{d/2}}e^{-\frac{|y-x|^2}{2(t-s)}}dy\cr
&\leq \frac{v_d(1)r^d}{(t-s)^{d/2}}.\nonumber
\end{align}
Thus, by the semigroup property of $\{P_t\}_{t\geq0}$, we have
$$
P_s\left[P_{t-s}\ind_{B(r)}(x)\right]^2\leq \left[ \frac{v_d(1)r^d}{(t-s)^{d/2}}\wedge 1\right]P_t\ind_{B(r)}(x),
$$
which yields
\begin{align}\label{dfsv156}
\int^t_0 P_s\left[P_{t-s}\ind_{B(r)}(x)\right]^2ds&\leq P_t\ind_{B(r)}(x)\left[\int^{t-r^2}_0\frac{r^d}{(t-s)^{d/2}}ds+\int^t_{t-r^2}1ds\right]\cr
&=P_t\ind_{B(r)}(x)\left[\int^{t}_{r^2}\frac{r^d}{u^{d/2}}du+r^2\right]\cr
&=P_t\ind_{B(r)}(x)\left[\frac{2}{d-2}(r^2-t^{1-\frac{d}{2}}r^d)+r^2\right]\cr
&\leq 3r^2P_t\ind_{B(r)}(x).
\end{align}
Plugging (\ref{fhegdb}) and (\ref{dfsv156}) into (\ref{vwevqv}) yields that
\begin{align}\label{gfdhfsbv}
\mathbb{P}_{\delta_x}(X_t(B(r))>0)&\geq\frac{P_t\ind_{B(r)}(x)}{P_t\ind_{B(r)}(x)+6r^2}\cr
&\geq \frac{e^{-3/2}3^{-d/2}v_d(1) r^dp(t/3,x)}{7r^2}\cr
&=c(d)r^{d-2}p(t/3,x),
\end{align}
where $c(d):=e^{-3/2}3^{-d/2}v_d(1)/7$ and the second inequality follows from the fact that for $t$ large enough,
\begin{align}
P_t\ind_{B(r)}(x)&=\int_{|y|\leq r}\frac{1}{(2\pi t)^{d/2}}e^{-\frac{|y-x|^2}{2t}}dy\cr
&\leq\frac{1}{(2\pi t)^{d/2}}v_d(1)r^d\cr
&<r^2.\nonumber
\end{align}
Plugging (\ref{gfdhfsbv}) into (\ref{trgetfvda}) yields that for $t$ large enough,
\begin{align}
\mathbb{P}(R_t\geq r)\leq e^{-c(d)r^{d-2}\int_{\mathbb{R}^d }p(t/3,x)dx}=e^{-c(d)r^{d-2}}\nonumber.
\end{align}
Thus, $\kappa_d\geq c_d>0$.
\qed

\textbf{Acknowledgements}
 The second author thanks Hui He for introducing the work of R\'ev\'esz \cite{reves02}, which planted the seed of the current paper. He also would like to thank Lina Ji and Jiawei Liu for useful discussions.

\end{document}